\documentclass{article}
\usepackage{amssymb,amsmath,amsthm, mathrsfs}
\usepackage{graphicx, float,color}
\usepackage[encapsulated]{CJK}
\usepackage[english]{babel}
\usepackage[colorlinks=true]{hyperref}
\hypersetup{urlcolor=blue, linkcolor=red, citecolor=blue}
\usepackage{geometry} 
\usepackage{array} 
\usepackage{paralist} 
\usepackage{verbatim} 
\usepackage[all]{xy}
\newtheorem{theorem}{Theorem}[section]

\newtheorem {proposition}[theorem]{Proposition}
\newtheorem {corollary}[theorem]{Corollary}

\newtheorem {definition}[theorem]{Definition}
\newtheorem {example}[theorem]{Example}

\usepackage{sectsty}
\allsectionsfont{\sffamily\mdseries\upshape} 

\title{Frequent, disjoint hypercyclicity and strong topological transitivity of generalized weighted shift operators on Hilbert $C^{\ast}$-modules}
\author{Song-Ung Ri, Hyon-Hui Ju, Jin-Myong Kim*}

\date{}
\begin{document}
\maketitle
{\makeatletter\renewcommand*\@makefnmark{}\footnotetext{

*Corresponding author E-mail address:
jm.kim0211@ryongnamsan.edu.kp(Jin-Myong Kim) }\makeatother}

\centerline{Faculty of Mathematics, {\bf Kim Il Sung} University,}

\centerline{Pyongyang, Democratic People's Republic of Korea}

\begin{abstract}
In this paper we study some dynamical properties such as Frequent Hypercyclicity Criterion, chaos, disjoint hypercyclicity and $\mathcal{F}$-transitivity via Furstenberg family $\mathcal{F}$ for generalized bilateral weighted shift operator on the standard Hilbert $C^{\ast}$-module over $C^{\ast}$-algebra of compact operators on a separable Hilbert space.

\end{abstract}



\section{Introduction}

Weighted shifts have been studied intensively as they serve as a good source of examples and counterexamples in the study of bounded
operators. The first example of a hypercyclic operator on a Banach space was given in 1969 by Rolewicz (\cite{R69}) for weighted shifts and every new notion in linear dynamics is first tested on weighted shifts.  Accordingly, most of the dynamical properties in linear dynamics (like hypercyclcity, topological mixing, chaos and etc.) were completely characterized for weighted shifts on  $\ell^{p}(\mathbb{N})$ and $\ell^{p}(\mathbb{Z})$ $(1\leq p< \infty)$(\cite{BMPP19}, \cite{CS04}, \cite{G00}, \cite{S95}).

This paper is concerned with the generalized bilateral weighted shift operator introduced by Ivkovi\'{c} (\cite{I24}) as follows: Let $\mathcal{H}$ be a separable Hilbert space and $B(\mathcal{H})$ be the $C^{\ast}$-algebra of all bounded linear operators on $\mathcal{H}$. And let $\mathcal{A}:=B_{0}(\mathcal{H})$ be the $C^{\ast}$-algebra of all compact operators on $\mathcal{H}$. Assume that $W:=(W_{j})_{j\in\mathbb{Z}}$ is a uniformly bounded sequence of invertible operators in $B(\mathcal{H})$ such that the sequence $\{W^{-1}_{j}\}_{j\in\mathbb{Z}}$ is also uniformly bounded in $B(\mathcal{H})$ and $U$ is a unitary operator on $\mathcal{H}$. Generalized bilateral weighted shift operator $T_{U, W}$ on $\ell_{2}(\mathcal{A})$, the standard right Hilbert module over $\mathcal{A}$, is defined by
\[(T_{U,W}(x))_{n}=W_{n}x_{n-1}U\]
for all $n\in\mathbb{Z}$ and $x:=(x_{j})_{j\in\mathbb{Z}}\in\ell_{2}(\mathcal{A})$.

In \cite{I24} was characterized hypercyclicity and obtained a sufficient condition for chaos for  the generalized bilateral weighted shift operator  $T_{U, W}$.

In this paper we study some dynamical properties of generalized bilateral weighted shift operator $T_{U, W}$  considered in \cite{I24}, such as Frequent Hypercyclicity Criterion, chaos, disjoint hypercyclicity and $\mathcal{F}$-transitivity via furstenberg family $\mathcal{F}$.

According to previous results, for the weighted shifts on $\ell^{p}(\mathbb{N})$ and $\ell^{p}(\mathbb{Z})$ $(1\leq p< \infty)$, hypercyclicity and weak mixing are equivalent (\cite{GM11}) and reiterative hypercyclcity, upper frequent hypercyclicity, frequent hypercyclicity and Devaney chaos are also equivalent (\cite{BR15}, \cite{BMPP16}).
And an invertible bilateral weighted shift on $\ell^{p}(\mathbb{Z})$ $(1\leq p< \infty)$ is chaotic if and only if it satisfies Frequent Hypercyclicity Criterion (\cite{BR15}).  Also any tuple of bilateral weighted shifts on $\ell^{p}(\mathbb{Z})$ $(1\leq p< \infty)$ fails to be disjoint weakly mixing and no tuple of bilateral weighted shifts on $\ell^{p}(\mathbb{Z})$ $(1\leq p< \infty)$ containing an invertible shift can be disjoint hypercyclic (\cite{BMS14}) while there exists a tuple of disjoint hypercyclic invertible pseudo-shifts on $\ell^{p}(\mathbb{Z})$ $(1\leq p< \infty)$ (\cite{R24}, \cite{BP07}). As for the hypercyclicity of the operators on
 $C^{\ast}$-algebra, in \cite{IT23} was characterized hypercyclic weighted composition operators on the
commutative $C^{\ast}$-algebra of continuous functions vanishing at infnity on a locally compact, non-compact
Hausdorff space and more references therein.

As the operator $T_{U, W}$ on $\ell_{2}(\mathcal{A})$ is a generalization of the bilateral weighted shift on $\ell^{p}(\mathbb{Z})$, one might wonder if the above mentioned properties are still satisfied for $T_{U, W}$.

On the other hand the study on recurrence properties for linear dynamical systems including hypercyclicity and transitivity have been more refined via furstenberg family $\mathcal{F}$ (see $\cite{AK21}$, $\cite{BG18}$, $\cite{BMPP16}$, $\cite{BMPP19}$ and $\cite{HH18}$).
In $\cite{BMPP16}$ was introduced $\mathcal{F}$-transitivity notion and characterized the $\mathcal{F}$-transitivity of weighted shifts on $\ell^{p}(\mathbb{Z})$.

The paper is organized as follows. In subsection 3.1 we give a sufficient condition for the operator $T_{U, W}$ to satisfy the Frequent Hypercyclicity Criterion, which improves the result in \cite{I24}. In particular we show that in some special cases, the operator  $T_{U, W}$ is chaotic if and only if it satisfies Frequent Hypercyclicity Criterion, which is similar to the case of weighted shifts on $\ell^{p}(\mathbb{Z})$.
In subsection 3.2  we show that as for the operator $T_{U, W}$, generalized weighted shifts on $\ell_{2}(\mathcal{A})$, we can get a tuple of disjoint hypercyclic weighted shifts  composed all invertible ones unlike the case of  weighted shifts on $\ell^{p}(\mathbb{Z})$. First  we give a characterization of disjoint hypercyclicity for a tuple of  the operators $T_{U,W}$ on $\ell_{2}(\mathcal{A})$ 
 and  illustrate it with a concrete example.
In subsection 3.3 we characterize $\mathcal{F}$-transitivity for $T_{U, W}$ and give an example of mixing operator $T_{U, W}$.

\section{Preliminaries}
Let $X$ denote a separable Banach space and $\mathcal{L}(X)$ be the algebra of continuous linear operators on $X$. From
now on, all operators on  $X$  will be in $\mathcal{L}(X)$, if nothing else is said. Denote $\mathbb{N}_{0}:=\mathbb{N}\cup \{0\}$.

A sequence $(T_{n})_{n\in\mathbb{N}_{0}}$ of operators on $X$ is called \textit{topologically transitive} (respectively, \textit{topologically mixing}) if for every pair of non-empty open sets $U, V\subset X$, the return set $N(U,V):=\{n\in\mathbb{N}_{0}: T_{n}(U)\cap V\neq \emptyset\}$ is infinite (respectively, cofinite). An operator $T$ on $X$ is called \textit{topologically transitive} (respectively, \textit{topologically mixing}) if the sequence $(T^{n})_{n\in\mathbb{N}_{0}}$ is topologically transitive (respectively, topologically mixing).
A sequence $(T_{n})_{n\in\mathbb{N}_{0}}$ of operators on $X$ is \textit{hypercyclic} if there exists a vector $x \in X$, called a \textit{hypercyclic vector}, such that its orbit $ Orb(x, (T_{n})) =\{T_{n}x:n\in {\mathbb{N}_{0}} \}$ is dense in $X$. This is equivalent to the fact that there is a vector $x \in X$ such that for any non-empty open set $U \subset X$ the return set $N(x,U) := \{n \in {\mathbb{N}_{0}}:T_{n}x \in U \}$ is non-empty. An operator $T$ on $X$ is called \textit{hypercyclic} if the sequence $(T^{n})_{n\in\mathbb{N}_{0}}$ is hypercyclic. It is well known that a sequence $(T_{n})_{n\in\mathbb{N}_{0}}$ of operators on a separable Banach space $X$ is hypercyclic if and only if it is topologically transitive. A hypercyclic operator $T$ on $X$ with a dense set of periodic points  is called \textit{chaotic}.

\subsection{Frequent hypercyclicity and disjoint hypercyclcity}

An operator $T$ on $X$ is said to be \textit{frequently hypercyclic} if there exists
a vector $x\in X$, called a \textit{frequently hypercyclic vector}, such that for every non-empty open subset $U \subset X$, the return set $N(x,U)$ has positive lower density, that is,

 \[\underline{d} (N(x,U))=\liminf_{n \rightarrow \infty}\frac{\# (N(x,U)\bigcap [0,n])}{n+1} >0, \]
 where $\# (\bullet)$ denotes the cardinality of the set $\bullet$.

Frequent Hypercyclicity Criterion is well known as a sufficient condition for frequent hypercyclicity.

\begin{theorem} \textnormal{(Frequent Hypercyclicity Criterion \cite{GM11})} Let $T$ be an operator on a separable Banach space $X$. If there is a dense subset $X_{0}$ of $X$ and a map $S:X_{0}\rightarrow X_{0}$ such that, for any $x\in X_{0}$,

$~~(i)~\sum^{\infty}_{n=0}T^{n}$ converges unconditionally,
	
$~(ii)~\sum^{\infty}_{n=0}S^{n}$ converges unconditionally,
	
$(iii)~TSx=x,$\\
then $T$ is frequently hypercyclic.
\end{theorem}

\begin{proposition} \textnormal{(Proposition 9.11, \cite{GM11})}
An operator on a separable Banach space that satisfies the Frequent Hypercyclicity Criterion is also chaotic and mixing.
\end{proposition}

Disjointness in hypercyclicity was introduced independently by Bernal \cite{B07}, and B\`{e}s and Peris \cite{BP07} in 2007.
For $N\geq 2$, the operators $T_{1}$, ..., $T_{N}\in \mathcal{L}(X)$ are called \textit{disjoint hypercyclic} if the direct sum $T_{1}\oplus \cdots \oplus T_{N}$ has a hypercyclic vector of the form $(x, ..., x)\in X^{N}$. Such a vector $x\in X$ is called a \textit{disjoint hypercyclic} vector for $T_{1}, ..., T_{N}$. If the set of disjoint hypercyclic vectors is dense in $X$, then the operators are called \textit{densely disjoint hypercyclic}.

And for $N\geq 2$, operators $T_{1}$, ..., $T_{N}\in \mathcal{L}(X)$ are called \textit{disjoint topologically transitive} if for any non empty open subsets $U, V_{1}, V_{2}, ..., V_{N}$ of $X$, there exists $n\in \mathbb{N}$ such that $U\cap T^{-n}_{1}(V_{1})\cap...\cap T^{-n}_{N}(V_{N})\neq \emptyset$.
The operators $T_{1}$, ..., $T_{N}\in \mathcal{L}(X)$ are disjoint topologically transitive if and only if the set of disjoint hypercyclic vectors is dense in $X$ (Proposition 2.3, \cite{BP07}).

\subsection{Furstenberg families}

A collection $\mathcal{F}$ of subsets of $\mathbb{N}_{0}$ is called a \textit{Furtenberg family} if it is hereditary upward, that is, $A\in\mathcal{F}$ and $A\subset B$ imply $B\in \mathcal{F}$.
We say that a Furstenberg family is \textit{proper} if it is non-empty and does not contain the empty set.
And a Furstenberg family $\mathcal{F}$ is called \textit{finitely invariant} if for any $A\in\mathcal{F}$ and all $n\geq 0$, $A\setminus [0,n]\in\mathcal{F}$.
The family of all sets containing infinitely (respectively, cofinitely) many positive integers, is denoted by $\mathcal{F}_{\rm inf}$ (respectively, $\mathcal{F}_{\rm cof}$).

Some authors defined the notion of limit along
a Furstenberg family $\mathcal{F}$(see \cite{BMPP19}). Given a sequence $\{ x_{n} \} $ and $x$ in $X$ we say that $\{ x_{n} \} $  ${\mathcal F}$-\textit{converges} to $x$ (denoted by $x_{n} \stackrel{\mathcal F}{\longrightarrow} x$) if for every neighborhood $V$ of $x$, $\{ n\in {\mathbb N}_{0}: x_{n} \in V\} \in {\mathcal F}$. If ${\mathcal F}={\mathcal F}_{{\rm cof}} $ then ${\mathcal F}$-convergence coincides with the ordinary convergence.

Now we recall the notion of topological transitivity via Furstenberg families.
\begin{definition}\textnormal{\cite{BMPP19}
Let $X$ be a separable infinite-dimensional Banach space and $\mathcal{F}$ be a proper Furstenberg family.
A sequence $(T_{n})_{n\in \mathbb{N}_{0}}$ of operators on $X$ is called \textit{$\mathcal{F}$-transitive} if for any non-empty open sets $U,V\subset X$,  $N(U,V):=\{n\in\mathbb{N}_{0}: T_{n}U\cap V\neq \emptyset\}\in \mathcal{F}$. An operator $T$ on $X$ is called \textit{$\mathcal{F}$-transitive} if the sequence $(T^{n})_{n\in\mathbb{N}_{0}}$ is $\mathcal{F}$-transitive.
}\end{definition}

A topologically transitive operator is $\mathcal{F}_{\rm inf}-$transitive and a topologically mixing operator is  $\mathcal{F}_{\rm cof}-$transitive operator.

\section{The main results}
Let $\mathcal{H}$ be a separable Hilbert space and $B(\mathcal{H})$ be the $C^{\ast}$-algebra of all bounded linear operators on $\mathcal{H}$. And let $\mathcal{A}:=B_{0}(\mathcal{H})$ be the $C^{\ast}$-algebra of all compact operators on $\mathcal{H}$. Assume that $W:=\{W_{j}\}_{j\in\mathbb{Z}}$ is a uniformly bounded sequence of invertible operators in $B(\mathcal{H})$ such that the sequence $\{W^{-1}_{j}\}_{j\in\mathbb{Z}}$ is also uniformly bounded in $B(\mathcal{H})$ and $U$ is a unitary operator on $\mathcal{H}$. Generalized bilateral weighted shift operator $T_{U, W}$ on $\ell_{2}(\mathcal{A})$, the standard right Hilbert module over $\mathcal{A}$, is defined by
\[(T_{U,W}(x))_{n}=W_{n}x_{n-1}U\]
for all $n\in\mathbb{Z}$ and $x:=(x_{j})_{j\in\mathbb{Z}}\in\ell_{2}(\mathcal{A})$.

$T_{U,W}$ is invertible and its inverse $S_{U,W}$ is given by
\[(S_{U,W}(y))_{n}=W^{-1}_{n+1}y_{n+1}U^{\ast}\]
for all $n\in\mathbb{Z}$ and $y:=(y_{j})_{j\in\mathbb{Z}}\in\ell_{2}(\mathcal{A})$ (see \cite{I24}).
Then we can see that
\[(T^{n}_{U,W}(x))_{i}=W_{i}W_{i-1}...W_{i-n+1}x_{i-n}U^{n}\]
and
\[(S^{n}_{U,W}(y))_{i}=W^{-1}_{i+1}W^{-1}_{i+2}...W^{-1}_{i+n}y_{i+n}U^{\ast n}\]
for all $n\in\mathbb{N}$, $i\in\mathbb{Z}$ and $x:=(x_{i})_{i}$, $y:=(y_{i})_{i}\in \ell_{2}(\mathcal{A})$. For more details see \cite{I24} and the references therein.

For each $J, m \in\mathbb{N}$,  we denote $[J]:=\{-J, -J+1,..., J\}$ and $L_{m}:=span\{e_{-m}, e_{-m+1},..., e_{m}\}$  where $\{e_{i}\}_{i\in\mathbb{Z}}$ is an orthonormal basis for $\mathcal{H}$, $P_{m}$ is the orthogonal projection onto $L_{m}$,
\[F_{J,m}:=\{x=(x_{j})_{j}\in\ell_{2}(\mathcal{A}):x_{j}=P_{m}x_{j} \textrm{ for } j\in [J], x_{j}=0 \textrm{ for } j\notin[J]\} \]
and $F:=\bigcup_{J\in \mathbb{N}, m\in\mathbb{N}}F_{J,m}$. We note that $F$ is dense in $\ell_{2}(\mathcal{A})$ by Proposition $2.2.1$ in $\cite{MT05}$.

\subsection{Frequent Hypercyclicity Criterion and Chaos for $T_{U,W}$}

In this section our study is concerned with Frequent Hypercyclicity Criterion and chaos for generalized bilateral weighted shift $T_{U,W}$ on $\ell_{2}(\mathcal{A})$.

Following proposition shows a sufficient condition for $T_{U,W}$ to satisfy Frequent Hypercyclic Criterion.
In \cite{I24} was obtained a sufficient condition for $T_{U,W}$ to be chaotic.
 Since the Frequent Hypercyclicity Criterion guarantees chaos as well as frequent hypercyclicity and topologically mixing,  the condition (2) of Proposition \ref{pro4.1} is also a sufficient condition for chaos.

\begin{proposition}\label{pro4.1}
Let $\mathcal{H}$ be a separable Hilbert space, $\mathcal{A}$ be the $C^{\ast}$-algebra of of compact operators on $\mathcal{H}$ and $T_{U,W}$ be the weighted shift operator on $\ell_{2}(\mathcal{A})$, the standard right Hilbert module over $\mathcal{A}$. Then we have $(2)\Rightarrow (1)$.

$(1)$ $T_{U, W}$ satisfies Frequent Hypercyclicity Criterion.

$(2)$ For every $J,m\in\mathbb{N}$ there exist a strictly increasing sequence $\{n_{k}\}_{k}\subset \mathbb{N}$ and a sequence $\{D^{(k)}_{i}\}_{k}$ for $i\in [J]$ of operators in $B_{0}(\mathcal{H})$ such that for all $j\in[J]$
\[\lim_{k\rightarrow\infty}\|D^{(k)}_{j}-P_{m}\|=0\]
and

$(i)$ $\sum^{\infty}_{l=1}\|W_{j+ln_{k}}\ldots W_{j+1}D^{(k)}_{j}\|^{2}$ converges for all $j\in[J]$, $k\in \mathbb{N}$,

$(ii)$ $\sum^{\infty}_{l=1}\|W^{-1}_{j-ln_{k}+1}\ldots W^{-1}_{j}D^{(k)}_{j}\|^{2}$ converges for all $j\in[J]$, $k\in \mathbb{N}$.
\end{proposition}

\begin{proof}
Let $J,m\in \mathbb{N}$. Assume that $\{n_{k}\}_{k}$ is a strictly increasing sequence and for every $i\in[J]$, $D^{(k)}_{i}$ is a sequence of operators on $B_{0}(\mathcal{H})$ satisfying the condition (2). We may assume $2J<n_{1}<n_{2}...$. We define
\[E_{J,m}:=\{z=(z_{j})_{j}\in \ell_{2}(\mathcal{A}): \textrm{ there exist } x=(x_{j})_{j}\in F_{J,m} \textrm{ and } k\in \mathbb{N} \textrm{ such that }\]
\[z_{j}=D^{(k)}_{j}x_{j} \textrm{ for } j\in [J], z_{j}=0 \textrm{ for } j\notin [J]\}\]
and $E:=\bigcup_{J,m\in\mathbb{N}}E_{J,m}$.
Then $F_{J,m}\subset \overline{E}_{J,m}$ ($\overline{E}_{J,m}$ is the closure of $E_{J,m}$). In fact for every $x\in F_{J,m}$, there exists a sequence $\{z^{(k)}\}_{k}\subset E_{J,m}$ which $z^{(k)}=(z^{(k)}_{j})_{j}$ is given by
\begin{displaymath}
z^{(k)}_{j}=\left\{ \begin{array}{ll}
 D^{(k)}_{j}x_{j}, & \textrm{for } j \in [J],\\
 0, & \textrm{else}.
\end{array} \right.
\end{displaymath}

Since $\lim_{k\rightarrow\infty}\|D^{(k)}_{j}-P_{m}\|=0$ for every $j\in[J]$, we have that
\begin{eqnarray}\begin{split}\nonumber
 \lim_{k\rightarrow\infty}\|z^{(k)}-x\|&\leq \lim_{k\rightarrow\infty}\sum_{j\in[J]}\|z^{(k)}_{j}-x_{j}\|
\\
&\leq\lim_{k\rightarrow\infty}\sum_{j\in[J]}\|D^{(k)}_{j}-P_{m}\|\|x_{j}\|=0,
\end{split}\end{eqnarray}
thus $\lim_{k\rightarrow\infty}z^{(k)}=x$. Therefore $E$ is dense in $\ell_{2}(\mathcal{A})$.

Now we show that $\sum^{\infty}_{n=1} T^{n}_{U,W}(z)$ and $\sum^{\infty}_{n=1} S^{n}_{U,W}(z)$ converge unconditionally for every $z\in E$. From the definition of E, for every $z\in E$ there exist $J,m\in\mathbb{N}$, $x=(x_{j})_{j}\in F_{J,m}$ and $k\in \mathbb{N}$ such that $z_{j}=D^{(k)}_{j}x_{j}$ for $j\in [J]$ and $z_{j}=0$ for $j\notin [J]$ in which $\{n_{k}\}_{k}$ and $\{D^{(k)}_{i}\}_{k}$ with $i\in[J]$ satisfy the condition (2). Since we can see that if $\sum^{\infty}_{l=1}T^{ln_{k}}_{U,W}(z)$ converges unconditionally, then $\sum^{\infty}_{n=1} T^{n}_{U,W}(z)$ converges unconditionally, we verify that for every sequence $\{l_{t}\}_{t}\subset \mathbb{N}$, $\sum^{\infty}_{t=1}T^{l_{t}n_{k}}_{U,W}(z)$ converges.
In fact, since $n_{k}>2J$, we have that
\begin{eqnarray}\begin{split}\nonumber
&\left(\sum^{\infty}_{t=1}T^{l_{t}n_{k}}_{U,W}(z)\right)_{j}=\sum^{\infty}_{t=1}\left(T^{l_{t}n_{k}}_{U,W}(z)\right)_{j}\\
&=\left\{ \begin{array}{ll}
T^{l_{s}n_{k}}_{U,W}(z)_{j}=W_{j}W_{j-1}\ldots W_{j+1-l_{s}n_{k}}z_{j-l_{s}n_{k}}U^{l_{s}n_{k}}, & \textrm{for some } s\in \mathbb{N}, j\in [J]+l_{s}n_{k},\\
0, & \textrm{else}.
\end{array} \right.
\end{split}\end{eqnarray}
Then it follows that
\begin{eqnarray}\begin{split}\nonumber
&\sum_{j\in\mathbb{Z}}{\left(\sum^{\infty}_{t=1}T^{l_{t}n_{k}}_{U,W}(z)\right)}^{\ast}_{j}\left(\sum^{\infty}_{t=1}T^{l_{t}n_{k}}_{U,W}(z)\right)_{j}\\
&=\sum^{\infty}_{s=1}\sum_{j\in[J]+l_{s}n_{k}}\left(\sum^{\infty}_{t=1}T^{l_{t}n_{k}}_{U,W}(z)\right)^{\ast}_{j}\left(\sum^{\infty}_{t=1}
T^{l_{t}n_{k}}_{U,W}(z)\right)_{j}\\
&=\sum^{\infty}_{s=1}\sum_{j\in[J]}\left(T^{l_{s}n_{k}}_{U,W}(z)\right)^{\ast}_{j+l_{s}n_{k}}\left(T^{l_{s}n_{k}}_{U,W}(z)\right)_{j+l_{s}n_{k}}\\
&=\sum_{j\in[J]}\sum^{\infty}_{s=1}U^{-l_{s}n_{k}}{x_{j}}^{\ast}{D^{(k)}_{j}}^{\ast}{W^{\ast}_{j+1}}\ldots {W^{\ast}_{j+l_{s}n_{k}}}W_{j+l_{s}n_{k}}\ldots W_{j+1}D^{(k)}_{j}x_{j}U^{l_{s}n_{k}}.
\end{split}\end{eqnarray}
For $j\in[J]$ and $m\in\mathbb{N}$, we have that
\begin{eqnarray}\begin{split}\nonumber
 \|U^{-l_{m}n_{k}}{x_{j}}^{\ast}&{D^{(k)}_{j}}^{\ast}{W^{\ast}_{j+1}}\ldots {W^{\ast}_{j+l_{m}n_{k}}}W_{j+l_{m}n_{k}}\ldots W_{j+1}D^{(k)}_{j}x_{j}U^{l_{m}n_{k}}\|\\
 &=\|W_{j+l_{m}n_{k}}\ldots W_{j+1}D^{(k)}_{j}x_{j}U^{l_{m}n_{k}}\|^{2}\\
 &\leq \|W_{j+l_{m}n_{k}}\ldots W_{j+1}D^{(k)}_{j}\|^{2}\|x\|^{2}.
\end{split}\end{eqnarray}
  Since $\sum^{\infty}_{l=1}\|W_{j+ln_{k}}\ldots W_{j+1}D^{(k)}_{j}\|^{2}$ converges from the condition $(i)$ of $(2)$,

\[\sum^{\infty}_{s=1}\|U^{-l_{s}n_{k}}{x_{j}}^{\ast}{D^{(k)}_{j}}^{\ast}{W^{\ast}_{j+1}}\ldots {W^{\ast}_{j+l_{s}n_{k}}}W_{j+l_{s}n_{k}}\ldots W_{j+1}D^{(k)}_{j}x_{j}U^{l_{s}n_{k}}\|\] converges, thus
\[\sum^{\infty}_{s=1}U^{-l_{s}n_{k}}{x_{j}}^{\ast}{D^{(k)}_{j}}^{\ast}{W^{\ast}_{j+1}}\ldots {W^{\ast}_{j+l_{s}n_{k}}}W_{j+l_{s}n_{k}}\ldots W_{j+1}D^{(k)}_{j}x_{j}U^{l_{s}n_{k}}\] also converges. Therefore $\sum^{\infty}_{t=1}T^{l_{t}n_{k}}_{U,W}(z)$ converges.
Similarly $\sum^{\infty}_{t=1}S^{l_{t}n_{k}}_{U,W}(z)$ converges. This completes the proof.
\end{proof}

We note that the condition $(2)$ in Proposition \ref{pro4.1} is weaker than the condition $(ii)$ of Proposition 3.4 in  \cite{I24}. Following example shows that Proposition \ref{pro4.1} improves  remarkably the Proposition 3.4 in  \cite{I24}.

\begin{example}\textnormal{
Let $\mathcal{H}$ be a separable Hilbert space and $W=\{W_{j}\}_{j\in\mathbb{Z}}$ be a sequence of operators in $B(\mathcal{H})$ which is defined
\begin{displaymath}
W_{i}(e_{j})=\left\{ \begin{array}{ll}
\frac{i}{i+1}e_{j+1}, &  \textrm{for } j\geq 0,\\
\frac{i+1}{i}e_{j+1}, & \textrm{for } j< 0,
\end{array} \right.
\end{displaymath}
by for $i>0$, $W_{i}=W^{-1}_{-i}$ for $i<0$ and $W_{0}=I$ is the identity operator. Let $J, m\in \mathbb{N}$.
Then for $i\geq 0$ and $l>m$, we have that
\begin{displaymath}
W_{i+l}W_{i+l-1}...W_{i+1}P_{m}(e_{j})=\left\{ \begin{array}{ll}
0, & \textrm{for } j>m, \\
\frac{i+1}{i+l+1}e_{j+l}, &  \textrm{for } m\geq j\geq 0,\\
\frac{(i-j+1)^{2}}{(i+1)(i+l+1)}e_{j+l}, &  \textrm{for } 0> j\geq -m,\\
0, & \textrm{for } -m>j,
\end{array} \right.
\end{displaymath}
thus $\|W_{i+l}W_{i+l-1}...W_{i+1}P_{m}\|=\frac{(i+m+1)^{2}}{(i+1)(i+l+1)}$
and for $i<0$ and $l>m-i$,
\begin{eqnarray}\begin{split}\nonumber
W_{i+l}W_{i+l-1}...W_{i+1}P_{m}(e_{j})&=W_{i+l}W_{i+l-1}...W_{-i}P_{m}(e_{j})\\
&=\left\{ \begin{array}{ll}
0, & \textrm{for } j>m, \\
\frac{-i}{i+l+1}e_{j+l+2i+1}, &  \textrm{for } m\geq j\geq 0,\\
\frac{(-i-j)^{2}}{(-i)(i+l+1)}e_{j+l+2i+1}, &  \textrm{for } 0> j\geq -m,\\
0, & \textrm{for } -m>j.
\end{array} \right.
\end{split}\end{eqnarray}
This means that
\[\|W_{i+l}W_{i+l-1}...W_{i+1}P_{m}\|=\frac{(m-i)^{2}}{(-i)(i+l+1)}.\]
Therefore $\sum^{\infty}_{l=1}\|W_{i+l}W_{i+l-1}...W_{i+1}P_{m}\|^{2}$ converges for every $i\in\mathbb{Z}$ and $m\in \mathbb{N}$. Similarly $\sum^{\infty}_{l=1}\|W^{-1}_{i-l+1}W^{-1}_{i-l+2}...W^{-1}_{i}P_{m}\|^{2}$ converges for every $i\in\mathbb{Z}$ and $m\in \mathbb{N}$. Setting $D^{(k)}_{j}=P_{m}$ for $j\in[J]$, $k\in\mathbb{N}$, we have that
 $\sum^{\infty}_{l=1}\|W_{j+ln_{k}}\ldots W_{j+1}D^{(k)}_{j}\|^{2}
=\sum^{\infty}_{l=1}\|W_{j+ln_{k}}\ldots W_{j+1}P_{m}\|^{2}$ and $\sum^{\infty}_{l=1}\|W^{-1}_{j-ln_{k}+1}\ldots W^{-1}_{j}D^{(k)}_{j}\|^{2}$ converge for all $j\in[J]$, $k\in \mathbb{N}$ and $n_{k}\in \mathbb{N}$. Thus $T_{U,W}$ satisfies Frequent Hypercyclicity Criterion from Proposition \ref{pro4.1}.}

\textnormal{
We note that $\sum^{\infty}_{l=1}\|W_{j+ln_{k}}\ldots W_{j+1}D^{(k)}_{j}\|$=$\sum^{\infty}_{l=1}\|W_{j+ln_{k}}\ldots
W_{j+1}P_{m}\|=O{( \sum^ {\infty}_{l=1}\frac{1}{j+ln_{k}})}$ diverges, which does not satisfy the condition $(ii)$ of Proposition 3.4 in \cite{I24}.
}
\end{example}

Following theorem shows in the case that $U$ is the identity $I\in B(\mathcal{H})$, some equivalent conditions for weighted shift $T_{I, W}$ on $\ell_{2}(\mathcal{A})$ to satisfy the Frequent Hypercyclicity Criterion, especially that $T_{I, W}$ on $\ell_{2}(\mathcal{A})$ is chaotic if and only if it satisfies the Frequent Hypercyclicity Criterion like the case of the invertible bilateral weighted shift on $\ell^{p}(\mathbb{Z})$(\cite{BR15}).

 We will denote the operator $T_{I,W}$ by $T_{W}$ for short.

\begin{theorem}\label{th2.4}
The following assertions are equivalent.

$(1)$ $T_{W}$ is chaotic.

$(2)$ $T_{W}$ satisfies Frequent Hypercyclicity Criterion.

$(3)$ For every $J,m\in\mathbb{N}$ there exist a strictly increasing sequence $\{n_{k}\}_{k}\in \mathbb{N}$ and a sequence $\{D^{(k)}_{i}\}_{k}$ for $i\in [J]$ of operators in $B_{0}(\mathcal{H})$ such that for all $j\in[J]$
\[\lim_{k\rightarrow\infty}\|D^{(k)}_{j}-P_{m}\|=0\]
and

$(i)$ $\sum^{\infty}_{l=1}{D^{(k)}_{j}}^{\ast}{W^{\ast}_{j+1}}\ldots {W^{\ast}_{j+ln_{k}}}W_{j+ln_{k}}\ldots W_{j+1}D^{(k)}_{j}$ converges for all $k\in \mathbb{N}$ and $j\in[J]$,

$(ii)$ $\sum^{\infty}_{l=1}{D^{(k)}_{j}}^{\ast}{{W^{-1}}^{\ast}_{j}}\ldots {{W^{-1}}^{\ast}_{j-ln_{k}+1}}W^{-1}_{j-ln_{k}+1}\ldots W^{-1}_{j}D^{(k)}_{j}$ converges for all $k\in \mathbb{N}$ and $j\in[J]$.

In particular if $T_{W}$ is chaotic, then it is frequently hypercyclic and mixing.
\end{theorem}
\begin{proof}
First, we show that $(1)\Rightarrow(3)$. Assume that $T_{W}$ is chaotic. And let $J,m\in\mathbb{N}$, and define $x=(x_{j})_{j}\in \ell_{2}{(\mathcal{A})}$ by $x_{j}:=P_{m}$ for all $j\in[J]$, and $x_{j}:=0$ for all $j\in{\mathbb{Z}\setminus[J]}$. Then for every $k\in \mathbb{N}$, there exist an element $y^{(k)}\in \ell_{2}{(\mathcal{A})}$ and $n_{k}\in\mathbb{N}$ such that $\|y^{(k)}-x\|\ \leq 2^{-k}$ and $T^{n_{k}}_{W}(y^{(k)})=y^{(k)}$.
Since  $T^{n_{k}}_{W}(y^{(k)})_{j}=W_{j}W_{j-1}\ldots W_{j+1-n_{k}}y^{(k)}_{j-n_{k}}$ for every $j\in\mathbb{Z}$, we have for every $j\in [J]$ and $l\in\mathbb{N}$ that
 \begin{eqnarray}\begin{split}\nonumber
 W_{j}W_{j-1}\ldots W_{j+1-ln_{k}}y^{(k)}_{j-ln_{k}}&=y^{(k)}_{j},\\
  y^{(k)}_{j-ln_{k}}&=W^{-1}_{j+1-ln_{k}}\ldots W^{-1}_{j-1}W^{-1}_{j}y^{(k)}_{j},\\
   y^{(k)}_{j+ln_{k}}&=W_{j+ln_{k}}\ldots W_{j+1}y^{(k)}_{j}.
 \end{split}\end{eqnarray}

Then we get that
\begin{eqnarray}\begin{split}\nonumber
\langle y^{(k)}, y^{(k)}\rangle &=\sum_{j\in\mathbb{Z}}{y^{(k){\ast}}_{j}}y^{(k)}_{j}\\
&=\sum^{n_{k}}_{i=1}\{{y^{(k){\ast}}_{i}}y^{(k)}_{i}+\sum^{\infty}_{l=1}{y^{(k){\ast}}_{i+ln_{k}}}y^{(k)}_{i+ln_{k}}+
\sum^{\infty}_{l=1}{y^{(k){\ast}}_{i-ln_{k}}}y^{(k)}_{i-ln_{k}}\\
&=\sum^{n_{k}}_{i=1}\{{y^{(k){\ast}}_{i}}y^{(k)}_{i}+\sum^{\infty}_{l=1}{y^{(k){\ast}}_{i}}{W^{\ast}_{i+1}}\ldots {W^{\ast}_{i+ln_{k}}}W_{i+ln_{k}}\ldots W_{i+1}y^{(k)}_{i}+\\
&+\sum^{\infty}_{l=1}{y^{(k){\ast}}_{i}}{W^{-1{\ast}}_{i}}\ldots {W^{-1{\ast}}_{i-ln_{k}+1}}W^{-1}_{i-ln_{k}+1}\ldots W^{-1}_{i}y^{(k)}_{i}\}.
\end{split}\end{eqnarray}
Since $\sum_{j\in\mathbb{Z}}{y^{(k){\ast}}_{j}}y^{(k)}_{j}$ converges and ${y^{(k){\ast}}_{j}}y^{(k)}_{j}$ is positive for every $k,j\in \mathbb{N}$,   
$\sum_{j\in\mathbb{Z}}{y^{(k){\ast}}_{j}}y^{(k)}_{j}$ converges unconditionally. Therefore we have that
\[\sum^{\infty}_{l=1}{y^{(k){\ast}}_{i}}{W^{\ast}_{i+1}}\ldots {W^{\ast}_{i+ln_{k}}}W_{i+ln_{k}}\ldots W_{i+1}y^{(k)}_{i}\]
and
\[\sum^{\infty}_{l=1}{y^{(k)}_{i}}^{\ast}{W^{-1{\ast}}_{i}}\ldots {W^{-1{\ast}}_{i-ln_{k}+1}}W^{-1}_{i-ln_{k}+1}\ldots W^{-1}_{i}y^{(k)}_{i}\]
 converge for all $i\in[J]$.
We conclude $(1)\Rightarrow(3)$ by defining $D^{(k)}_{j}:=y^{(k)}_{j}$ for all $j\in[J]$ and $k\in\mathbb{N}$.

$(2)\Rightarrow(1)$ is trivial. Now it is sufficient to show that $(3)\Rightarrow(2)$.
 In the same way with the proof Propostion \ref{pro4.1}, we get that
\begin{eqnarray}\begin{split}\nonumber
\sum_{j\in\mathbb{Z}}&\left(\sum^{\infty}_{t=1}T^{l_{t}n_{k}}_{W}(z)\right)^{\ast}_{j}\left(\sum^{\infty}_{t=1}T^{l_{t}n_{k}}_{W}(z)\right)_{j}=\\
&=\sum_{j\in[J]}\sum^{\infty}_{m=1}{x^{\ast}_{j}}{D^{(k){\ast}}_{j}}{W^{\ast}_{j+1}}\ldots {W^{\ast}_{j+l_{m}n_{k}}}W_{j+l_{m}n_{k}}\ldots W_{j+1}D^{(k)}_{j}x_{j}\\
&=\sum_{j\in[J]}{x^{\ast}_{j}}\left(\sum^{\infty}_{m=1}{D^{(k){\ast}}_{j}}{W^{\ast}_{j+1}}\ldots {W^{\ast}_{j+l_{m}n_{k}}}W_{j+l_{m}n_{k}}\ldots W_{j+1}D^{(k)}_{j}\right)x_{j}.
\end{split}\end{eqnarray}

Since the operator ${D^{(k){\ast}}_{j}}{W^{\ast}_{j+1}}\ldots {W^{\ast}_{j+ln_{k}}}W_{j+ln_{k}}\ldots W_{j+1}D^{(k)}_{j}$ is positive for every $k,j\in\mathbb{N}$, $\sum^{\infty}_{m=1}{D^{(k){\ast}}_{j}}{W^{\ast}_{j+1}}\ldots {W^{\ast}_{j+l_{m}n_{k}}}W_{j+l_{m}n_{k}}\ldots W_{j+1}D^{(k)}_{j}$ converges unconditionally from the condition (i) of (3). Therefore $\sum^{\infty}_{t=1}T^{l_{t}n_{k}}_{W}(z)$ converges, thus $\sum^{\infty}_{n=1} T^{n}_{W}(z)$ converges unconditionally. Similarly $\sum^{\infty}_{n=1} S^{n}_{W}(z)$ converges unconditionally, which concludes the proof.
\end{proof}

\subsection{Disjoint hypercyclicity for $T_{U,W}$}

It is known that  any tuple of  bilateral shifts  on $\ell^{p}(\mathbb{Z})$ containing an invertible shift fails to be disjoint hypercyclic  (\cite{BMS14}).  In this subsection
the we show that as for the generalized weighted shifts on $\ell_{2}(\mathcal{A})$ we can get a tuple of disjoint hypercyclic weighted shifts  composed all invertible ones. We can see in \cite{R24} and \cite{BP07} that as for the pseudo-shifts on $\ell^{p}(\mathbb{Z})$, there exists a
tuple of disjoint hypercyclic pseudo-shifts containing invertible ones.

The following theorem characterizes disjoint hypercyclicity of the generalized bilateral weighted shift operators $T_{U,W}$ on $\ell_{2}(\mathcal{A})$. We note that $T_{U,W}$ is invertible.

\begin{theorem}\label{the5.1}
Let $W^{(1)}$, $W^{(2)}$, ..., $W^{(N)}$ be uniformly bounded sequences of invertible operators in $B(\mathcal{H})$ and $U^{(1)}$, $U^{(2)}$, ..., $U^{(N)}$ be unitary operators in $B(\mathcal{H})$. Assume that for each $m\in \mathbb{N}$ there exists an $N_{m}\in\mathbb{N}$ such that for each $n\geq N_{m}$ and each pair of distinct $s,l\in \{1,2, ..., N\}$
\begin{equation} \label{eq1}
{U^{(s)}}^{n}{U^{(l)}}^{-n}(L_{m})\perp L_{m}\tag{*}.
\end{equation}
Then the following assertions are equivalent.

$(1)$ The operators $T_{U^{(1)},W^{(1)}}$, $T_{U^{(2)},W^{(2)}}$, ..., $T_{U^{(N)},W^{(N)}}$ are densely disjoint hypercyclic.

$(2)$ For every $J,m \in\mathbb{N}$, there exist a strictly increasing sequence $\{n_{k}\}_{k}\subset\mathbb{N}$ and sequences
$\{D^{(k)}_{j}\}^{\infty}_{k=1}$ and $\{G^{(k)}_{1,j}\}^{\infty}_{k=1}$, ..., $\{G^{(k)}_{N,j}\}^{\infty}_{k=1}$
 for all $j\in [J]$ of operators in $B_{0}(\mathcal{H})$ such that for each $ j\in [J]$ and $l\in\{1,2,...,N\}$,
\[\lim_{k\rightarrow\infty}\|D^{(k)}_{j}-P_{m}\|=0, \lim_{k\rightarrow\infty}\|G^{(k)}_{l,j}-P_{m}\|=0,\]
\[\lim_{k\rightarrow\infty}\|W^{(l)}_{j+n_{k}}W^{(l)}_{j+n_{k}-1},...,W^{(l)}_{j+1}D^{(k)}_{j}\|=0,
\lim_{k\rightarrow\infty}\|{W^{(l)}_{j-n_{k}+1}}^{-1}{W^{(l)}_{j-n_{k}+2}}^{-1},...,{W^{(l)}_{j}}^{-1}G^{(k)}_{l,j}\|=0\]
 and, for each $j\in[J]$ and each pair of distinct $s,l\in \{1,2,,,N\}$,
 \[\lim_{k\rightarrow\infty}\|W^{(s)}_{j},...,W^{(s)}_{j-n_{k}+1}{W^{(l)}_{j-n_{k}+1}}^{-1}{W^{(l)}_{j-n_{k}+2}}^{-1},...,
 {W^{(l)}_{j}}^{-1}G^{(k)}_{l,j}\|=0.\]
\end{theorem}
\begin{proof}
First we show $(1)\Rightarrow(2)$. Assume that the operators $T_{U^{(1)},W^{(1)}}$, $T_{U^{(2)},W^{(2)}}$, ..., $T_{U^{(N)},W^{(N)}}$ are densely disjoint hypercyclic.
Let $J,m\in\mathbb{N}$ and define $x=(x_{j})_{j}\in \ell_{2}{(\mathcal{A})}$ by $x_{j}:=P_{m}$ for all $j\in[J]$ and $x_{j}:=0$ for all $j\in \mathbb{Z}\setminus[J]$. Then for every $k\in \mathbb{N}$, there exist an element $y^{(k)}\in \ell_{2}{(\mathcal{A})}$ and $n_{k}\in\mathbb{N}$ such that $\|y^{(k)}-x\|\ \leq 2^{k}$ and $\|T^{n_{k}}_{U^{(l)},W^{(l)}}(y^{(k)})-x\|\leq 2^k$ for each $l\in \{1,2,...,N\}$.

We may assume that the sequence $(n_{k})_{k}$ is strictly increasing, and $\max\{2J,N_{m}\}<n_{1}<n_{2}<\cdots$.

Set
\[D^{(k)}_{j}:=y^{(k)}_{j}P_{m}, G^{(k)}_{l,j}:=W^{(l)}_{j}W^{(l)}_{j-1}...W^{(l)}_{j-n_{k}+1}y^{(k)}_{j-n_{k}}{U^{(l)}}^{n_{k}}P_{m}\]
for each $j\in[J]$ and $l\in \{1,2,...,N\}$. Then we have for $j\in[J]$ and $l\in\{1,2,...,N\}$ that

\begin{eqnarray}\begin{split}\nonumber
\lim_{k\rightarrow\infty}\|D^{(k)}_{j}-P_{m}\|
&=\lim_{k\rightarrow\infty}\|y^{(k)}_{j}P_{m}-P_{m}\|\\
&\leq \lim_{k\rightarrow\infty}\|y^{(k)}_{j}-P_{m}\|\|P_{m}\|\\
&\leq\lim_{k\rightarrow\infty}\|y^{(k)}-x\|=0,
\end{split}\end{eqnarray}

\begin{eqnarray}\begin{split}\nonumber
\lim_{k\rightarrow\infty}\|G^{(k)}_{l,j}-P_{m}\|
&=\lim_{k\rightarrow\infty}\|W^{(l)}_{j}W^{(l)}_{j-1}...W^{(l)}_{j-n_{k}+1}y^{(k)}_{j-n_{k}}{U^{(l)}}^{n_{k}}P_{m}-P_{m}\|\\
&=\lim_{k\rightarrow\infty}\|\left(T^{n_{k}}_{U^{(l)},W^{(l)}}(y^{(k)})\right)_{j}P_{m}-P_{m}\|\\
&\leq \lim_{k\rightarrow\infty}\|\left(T^{n_{k}}_{U^{(l)},W^{(l)}}(y^{(k)})\right)_{j}-P_{m}\|\\
&\leq \lim_{k\rightarrow\infty}\|T^{n_{k}}_{U^{(l)},W^{(l)}}(y^{(k)})-x\|=0,
\end{split}\end{eqnarray}
\begin{eqnarray}\begin{split}\nonumber
\lim_{k\rightarrow\infty}\|W^{(l)}_{j+n_{k}}W^{(l)}_{j+n_{k}-1},...,W^{(l)}_{j+1}D^{(k)}_{j}\|&= \lim_{k\rightarrow\infty}\|\left(T^{n_{k}}_{U^{(l)},W^{(l)}}(y^{(k)})\right)_{j+n_{k}}{U^{(l)}}^{-n_{k}}P_{m}\| \\ &\leq\lim_{k\rightarrow\infty}\|\left(T^{n_{k}}_{U^{(l)},W^{(l)}}(y^{(k)})\right)_{j+n_{k}}\|\\
&\leq \lim_{k\rightarrow\infty}\|T^{n_{k}}_{U^{(l)},W^{(l)}}(y^{(k)})-x\|=0
\end{split}\end{eqnarray}
and
\begin{eqnarray}\begin{split}\nonumber
\lim_{k\rightarrow\infty}\|{W^{(l)}_{j-n_{k}+1}}^{-1}{W^{(l)}_{j-n_{k}+2}}^{-1},...,{W^{(l)}_{j}}^{-1}G^{(k)}_{l,j}\|
&=\lim_{k\rightarrow\infty}\|y^{(k)}_{j-n_{k}}{U^{(l)}}^{n_{k}}P_{m}\|\\
&\leq\lim_{k\rightarrow\infty}\|y^{(k)}_{j-n_{k}}\|\\
&\leq \lim_{k\rightarrow\infty}\|y^{(k)}-x\|=0.
\end{split}\end{eqnarray}

Since for each $n\geq N_{m}$ and each pair of distinct $s,l\in \{1,2,...,N\}$,  ${U^{(s)}}^{n}{U^{(l)}}^{-n}(L_{m})\perp L_{m}$, we have  for each pair of distinct $s,l\in \{1,2,...,N\}$ that
\[P_{m}{U^{(s)}}^{n}{U^{(l)}}^{-n}P_{m}=0.\]
Then it follows that

\begin{eqnarray}\begin{split}\nonumber
&\lim_{k\rightarrow\infty}\|W^{(s)}_{j},...,W^{(s)}_{j-n_{k}+1}{W^{(l)}_{j-n_{k}+1}}^{-1}{W^{(l)}_{j-n_{k}+2}}^{-1},...,
{W^{(l)}_{j}}^{-1}G^{(k)}_{l,j}\|\\
&=\lim_{k\rightarrow\infty}\|\left(T^{n_{k}}_{U^{(s)},W^{(s)}}(y^{(k)})\right)_{j} {U^{(s)}}^{-n}{U^{(l)}}^{n}P_{m}\|\\
&=\lim_{k\rightarrow\infty}\|\left(T^{n_{k}}_{U^{(s)},W^{(s)}}(y^{(k)})\right)_{j} {U^{(s)}}^{-n}{U^{(l)}}^{n}P_{m}-P_{m}{U^{(s)}}^{-n}{U^{(l)}}^{n}P_{m}\|\\
&\leq\lim_{k\rightarrow\infty}\|\left(T^{n_{k}}_{U^{(s)},W^{(s)}}(y^{(k)})\right)_{j}-P_{m}\|\\
&\leq \lim_{k\rightarrow\infty}\|T^{n_{k}}_{U^{(s)},W^{(s)}}(y^{(k)})-x\|=0,\\
\end{split}\end{eqnarray}
which completes the proof of $(1)\Rightarrow(2)$.

Next, we show $(2)\Rightarrow(1)$. In order to show this, it suffices to show that the operators $T_{U^{(1)},W^{(1)}}$, $T_{U^{(2)},W^{(2)}}$, ..., $T_{U^{(N)},W^{(N)}}$ are disjoint topologically transitive. Let $O, V_{1}, V_{2}, ...,  V_{N}$ be nonempty open subsets of $\ell_{2}(\mathcal{A})$. Since $F$ is dense in $\ell_{2}(\mathcal{A})$, we can find some $x=(x_{j})_{j}\in O$ and $y^{(1)}=(y^{(1)}_{j})_{j}\in V_{1}$, ..., $y^{(N)}=(y^{(N)}_{j})_{j}\in V_{N}$ and sufficiently large $J, m\in\mathbb{N}$ such that $x_{j}=P_{m}x_{j}$ and $y^{(l)}_{j}=P_{m}y^{(l)}_{j}$ for all $l\in\{1,2,...,N\}$,   $j\in[J]$ and $x_{j}=0$ and $y^{(l)}_{j}=0$ for all $l\in\{1,2,...,N\}$, $j\in \mathbb{Z}\setminus[J]$. Let the strictly increasing sequence $\{n_{k}\}_{k}\subset\mathbb{N}$ and sequences $\{D^{(k)}_{j}\}^{\infty}_{k=1}$ and $\{G^{(k)}_{1,j}\}^{\infty}_{k=1}$, ..., $\{G^{(k)}_{N,j}\}^{\infty}_{k=1}$ satisfy (ii) for these $J,m\in\mathbb{N}$. For each k, define $u^{(k)}, v^{(k)}_{1}, ..., v^{(k)}_{N}\in \ell_{2}(\mathcal{A})$ by $(u^{(k)})_{j}:=D^{(k)}_{j}x_{j}$ for $j\in[J]$, $(u^{(k)})_{j}:=0$ for $j\in \mathbb{Z}\setminus [J]$, and for each $l\in\{1,2,...,N\}$, $(v^{(k)}_{l})_{j}:=G^{(k)}_{l, j}y^{(l)}_{j}$ for $j\in[J]$, $(v^{(k)}_{l})_{j}:=0$ for $j\in \mathbb{Z}\setminus [J]$. Since $\lim_{k\rightarrow\infty}\|D^{(k)}_{j}-P_{m}\|=0$ and $\lim_{k\rightarrow\infty}\|G^{(k)}_{l,j}-P_{m}\|=0$  for each $j\in [J]$ and $l\in\{1,2,...,N\}$, $\lim_{k\rightarrow\infty}u^{(k)}=x$ and $\lim_{k\rightarrow\infty}v^{(k)}_{l}=y^{(l)}$. Set
\[\varphi_{k}:=u^{(k)}+\sum^{N}_{l=1}S^{n_{k}}_{U^{(l)},W^{(l)}}v^{(k)}_{l}.\]
Then we have that
\begin{eqnarray}\begin{split}\nonumber
\lim_{k\rightarrow\infty}\|\varphi_{k}-x\|&\leq \lim_{k\rightarrow\infty}\|u^{(k)}-x\|+ \lim_{k\rightarrow\infty}\sum^{N}_{l=1}\|S^{n_{k}}_{U^{(l)},W^{(l)}}v^{(k)}_{l}\|\\
&=\lim_{k\rightarrow\infty}\sum^{N}_{l=1}\|S^{n_{k}}_{U^{(l)},W^{(l)}}v^{(k)}_{l}\|\\
&\leq
\lim_{k\rightarrow\infty}\sum^{N}_{l=1}\sum_{j\in\mathbb{Z}}\|(S^{n_{k}}_{U^{(l)},W^{(l)}}v^{(k)}_{l})_{j}\|\\
&=\lim_{k\rightarrow\infty}\sum^{N}_{l=1}\sum_{j\in[J]}\|(S^{n_{k}}_{U^{(l)},W^{(l)}}v^{(k)}_{l})_{j-n_{k}}\|\\
&=\lim_{k\rightarrow\infty}\sum^{N}_{l=1}\sum_{j\in[J]}\|{W^{(l)}_{j-n_{k}+1}}^{-1},...,{W^{(l)}_{j}}^{-1} G^{(k)}_{l,j}y^{(l)}_{j}\|\\
&\leq\lim_{k\rightarrow\infty}\sum^{N}_{l=1}\sum_{j\in[J]}\|{W^{(l)}_{j-n_{k}+1}}^{-1},...,{W^{(l)}_{j}}^{-1} G^{(k)}_{l,j}\| \|y^{(l)}_{j}\|=0
\end{split}\end{eqnarray}
and

\begin{eqnarray}\begin{split}\nonumber
\lim_{k\rightarrow\infty}\|&T^{n_{k}}_{U^{(l)},W^{(l)}}(\varphi_{k})-y^{(l)}\|=\\
&\leq \lim_{k\rightarrow\infty}\|T^{n_{k}}_{U^{(l)},W^{(l)}}(u^{(k)})\|+ \lim_{k\rightarrow\infty}\|v^{(k)}_{l}-y^{(l)}\|\\
&+\lim_{k\rightarrow\infty}\sum^{N}_{s\neq l, s=1}\|T^{n_{k}}_{U^{(l)},W^{(l)}}\left(S^{n_{k}}_{U^{(s)},W^{(s)}}(v^{(k)}_{s})\right)\|
\leq \lim_{k\rightarrow\infty}\sum_{j\in\mathbb{Z}}\|\left(T^{n_{k}}_{U^{(l)},W^{(l)}}(u^{(k)})\right)_{j}\|\\
&+ \lim_{k\rightarrow\infty}\sum^{N}_{s\neq l, s=1}\sum_{j\in\mathbb{Z}}\|\left(T^{n_{k}}_{U^{(l)},W^{(l)}}\left(S^{n_{k}}_{U^{(s)},W^{(s)}}(v^{(k)}_{s})\right)\right)_{j}\|\\
&= \lim_{k\rightarrow\infty}\sum_{j\in[J]}\|W^{(l)}_{j+n_{k}}W^{(l)}_{j+n_{k}-1},...,W^{(l)}_{j+1}D^{(k)}_{j}x_{j}\|\\
&+\lim_{k\rightarrow\infty}\sum^{N}_{s\neq l, s=1}\sum_{j\in[J]}\|W^{(l)}_{j},...,W^{(l)}_{j-n_{k}+1}{W^{(s)}_{j-n_{k}+1}}^{-1}{W^{(s)}_{j-n_{k}+2}}^{-1},...,
 {W^{(s)}_{j}}^{-1}G^{(k)}_{s,j}y^{(s)}_{j}\|\\
&\leq \lim_{k\rightarrow\infty}\sum_{j\in[J]}\|W^{(l)}_{j+n_{k}}W^{(l)}_{j+n_{k}-1},...,W^{(l)}_{j+1}D^{(k)}_{j}\|\|x_{j}\|\\
&+\lim_{k\rightarrow\infty}\sum^{N}_{s\neq l, s=1}\sum_{j\in[J]}\|W^{(l)}_{j},...,W^{(l)}_{j-n_{k}+1}{W^{(s)}_{j-n_{k}+1}}^{-1}{W^{(s)}_{j-n_{k}+2}}^{-1},...,
 {W^{(s)}_{j}}^{-1}G^{(k)}_{s,j}\|\|y^{(s)}_{j}\|=0.
\end{split}\end{eqnarray}

It follows that $\lim_{k\rightarrow\infty}\varphi_{k}=x$ and $\lim_{k\rightarrow\infty}T^{n_{k}}_{U^{(l)},W^{(l)}}(\varphi_{k})=y^{(l)}$ for $l\in\{1,2,...,N\}$. Hence the operators $T_{U^{(1)},W^{(1)}}$, $T_{U^{(2)},W^{(2)}}$, ..., $T_{U^{(N)},W^{(N)}}$ are disjoint topologically transitive.
\end{proof}

In fact, in the proof of Theorem \ref{the5.1}, $(2)\Rightarrow(1)$ shows that if the sequences $W^{(1)},...,W^{(N)}$ satisfy the condition $(2)$, then $T_{U^{(1)},W^{(1)}}$, $T_{U^{(2)},W^{(2)}}$, ..., $T_{U^{(N)},W^{(N)}}$ are disjoint topologically transitive for any unitary operators $U^{(1)},...,U^{(N)}$ on $\mathcal{H}$.
This leads us to deduce the following.

\begin{corollary}
Let $\tilde{U}^{(1)},\tilde{U}^{(2)}, ..., \tilde{U}^{(N)}$ be unitary operators on $\mathcal{H}$ satisfying the condition $(\ast)$. Then $T_{\tilde{U}^{(1)},W^{(1)}}$, $T_{\tilde{U}^{(2)},W^{(2)}}$, ..., $T_{\tilde{U}^{(N)},W^{(N)}}$ are disjoint topologically transitive if and only if $T_{U^{(1)},W^{(1)}}$, $T_{U^{(2)},W^{(2)}}$, ..., $T_{U^{(N)},W^{(N)}}$ are disjoint topologically transitive for any unitary operators $U^{(1)}, ..., U^{(N)}$ on $\mathcal{H}$. In other words, $T_{U^{(1)},W^{(1)}}$, $T_{U^{(2)},W^{(2)}}$, ..., $T_{U^{(N)},W^{(N)}}$ are disjoint topologically transitive for any unitary operators $U^{(1)}, ..., U^{(N)}$ on $\mathcal{H}$ if and only if the operators $W^{(1)}, ..., W^{(N)}$ satisfy the condition $(2)$ in Theorem \ref{the5.1}.
\end{corollary}

The following example shows a tuple of disjoint hypercyclic weighted shifts on $\ell_{2}(\mathcal{A})$, composed by invertible ones.

\begin{example} \textnormal{
Let $\mathcal{H}$ be a separable Hilbert space and $(e_{j})_{j\in\mathbb{Z}}\subset\mathcal{H}$ be an orthonormal basis for $\mathcal{H}$. From Example 1 in \cite{IT21}, we can take unitary operators $U^{(1)},U^{(2)}$ on $\mathcal{H}$ satisfying condition \eqref{eq1}.
And let $W_{1}$ and $W_{2}$ be bounded linear operators on $\mathcal{H}$ defined by
\begin{displaymath}
W_{1}(e_{n})=\left\{ \begin{array}{ll}
 2e_{n+1}, & \textrm{for } n<0,\\
 \frac{1}{2}e_{n+1}, & \textrm{for } n\geq 0,
\end{array} \right.
\end{displaymath}
\begin{displaymath}
W_{2}(e_{n})=\left\{ \begin{array}{ll}
 3e_{n+1}, & \textrm{for } n<0,\\
 \frac{1}{3}e_{n+1}, & \textrm{for } n\geq 0.
\end{array} \right.
\end{displaymath}
Now we define $W^{(1)}_{j}=W_{1}$, $W^{(2)}_{j}=W^{2}_{2}$ for every $j\in \mathbb{Z}$.
By the similar arguments to Example 2.9 in \cite{IT23dis}, we have that}

\begin{eqnarray}\begin{split}\nonumber
\lim_{k\rightarrow\infty}\|W^{(1)}_{j+n_{k}}W^{(1)}_{j+n_{k}-1},...,W^{(1)}_{j+1}P_{m}\|&= \lim_{k\rightarrow\infty}\|W^{n_{k}}_{1}P_{m}\|=0,\\
\lim_{k\rightarrow\infty}\|{W^{(1)}_{j-n_{k}+1}}^{-1}{W^{(1)}_{j-n_{k}+2}}^{-1},...,{W^{(1)}_{j}}^{-1}P_{m}\|&= \lim_{k\rightarrow\infty}\|W^{-n_{k}}_{1}P_{m}\|=0,\\
\lim_{k\rightarrow\infty}\|W^{(2)}_{j+n_{k}}W^{(2)}_{j+n_{k}-1},...,W^{(2)}_{j+1}P_{m}\|&= \lim_{k\rightarrow\infty}\|W^{2n_{k}}_{2}P_{m}\|=0,\\
\lim_{k\rightarrow\infty}\|{W^{(2)}_{j-n_{k}+1}}^{-1}{W^{(2)}_{j-n_{k}+2}}^{-1},...,{W^{(2)}_{j}}^{-1}P_{m}\|&= \lim_{k\rightarrow\infty}\|W^{-2n_{k}}_{2}P_{m}\|=0,
\end{split}\end{eqnarray}
\textnormal{
\[\lim_{k\rightarrow\infty}\|W^{(1)}_{j},...,W^{(1)}_{j-n_{k}+1}{W^{(2)}_{j-n_{k}+1}}^{-1}{W^{(2)}_{j-n_{k}+2}}^{-1},...,
 {W^{(2)}_{j}}^{-1}P_{m}\|=\lim_{k\rightarrow\infty}\|W^{n_{k}}_{1}W^{-2n_{k}}_{2}P_{m}\|=0\]
 and
\[\lim_{k\rightarrow\infty}\|W^{(2)}_{j},...,W^{(2)}_{j-n_{k}+1}{W^{(1)}_{j-n_{k}+1}}^{-1}{W^{(1)}_{j-n_{k}+2}}^{-1},...,
 {W^{(1)}_{j}}^{-1}P_{m}\|=\lim_{k\rightarrow\infty}\|W^{2n_{k}}_{2}W^{-n_{k}}_{1}P_{m}\|=0.\]
Hence it follows that the operators $T_{U^{(1)},W^{(1)}}$, $T_{U^{(2)},W^{(2)}}$ are densely disjoint hypercyclic though they are all invertible.
}
\end{example}

\subsection{$\mathcal{F}$-transitivity for $T_{U,W}$}

Ivkovi\'{c} (\cite{I24}) characterized topological transitivity of $T_{U, W}$. In this section we consider the its $\mathcal{F}$-version for Furstenberg family $\mathcal{F}$.

First, we characterize $\mathcal{F}$-transitive weighted shift operator $T_{U, W}$ on $\ell_{2}(\mathcal{A})$ and thus extend the result of Ivkovi\'{c} (\cite{I24}, Proposition 3.1).

Here we define a notion of uniform $\mathcal{F}$-convergence.

\begin{definition}\textnormal{
Let $I$ be a set, $X_{i}$ be a separable Banach space  for every $i\in I$ and $\mathcal{F}$ be a Furstenberg family.
Assume that  $\{x^{(i)}_{j}\}_{j}\subset X_{i}$ is a sequence of vectors in $X_{i}$ for every $i\in I$. Then $\{x^{(i)}_{j}\}_{j}$ is said to \textit{uniformly $\mathcal{F}$-converge to} $x^{(i)}\in X_{i}$ \textit{for} $i\in I$ (denoted by $x^{(i)}_{j}\stackrel{\mathcal F}{\rightrightarrows} x^{(i)}$, as $j\rightarrow \infty$ for $i\in I$) if for every neighborhood $V_{i}\subset X_{i}$ of $x^{(i)}$, $i\in I$, there exists a set $F\in \mathcal{F}$ such that  $F\subset\{j\in \mathbb{N}:x^{(i)}_{j}\in V_{i}\}$ for any $i\in I$.}
\end{definition}

\begin{proposition}\label{pro3.1}
Let $\mathcal{F}$ be a finitely invariant Furstenberg family. And let $(t_{n})_{n}$ be an unbounded sequence of nonnegative integers. We denote $T_{U,W,n}:=T^{t_{n}}_{U,W}$ for all $n\in \mathbb{N}$. Then the following assertions are equivalent.

$(1)$ $(T_{U,W,n})_{n}$ is $\mathcal{F}$-transitive.

$(2)$ For every $J,m\in\mathbb{N}$ there exist sequences $\{D^{(n)}_{j}\}_{n\in\mathbb{N}}$ and $\{G^{(n)}_{j}\}_{n\in\mathbb{N}}$  of operators in $B_{0}(\mathcal{H})$ for all $i\in[J]$ such that
\[\left(\|D^{(n)}_{j}-P_{m}\|, \|G^{(n)}_{j}-P_{m}\|, \|W_{j+t_{n}}W_{j+t_{n}-1}...W_{j+1}D^{(n)}_{j}\|, \|W^{-1}_{j-t_{n}+1}W^{-1}_{j-t_{n}+2}...W^{-1}_{j}G^{(n)}_{j}\|\right)\]
\[\stackrel{\mathcal F}{\rightrightarrows} (0,0,0,0), \textrm{ as }  n\rightarrow\infty, \textrm{ for } j\in [J].\]
\end{proposition}

\begin{proof}
First, we show that $(1)\Rightarrow(2)$. Assume that $J,m\in \mathbb{N}$, and define $x=(x_{j})_{j}\in \ell_{2}(\mathcal{A})$ by $x_{j}:=P_{m}$ for all $j\in [J]$, and $x_{j}:=0$ for all $j\in \mathbb{Z}\setminus[J]$. Then there exists a sequence $\{y^{(n)}\}_{n}\subset \ell_{2}(\mathcal{A})$ such that
\[\left(y^{(n)},T^{t_{n}}_{U,W}(y^{(n)})\right)\stackrel{\mathcal F}{\longrightarrow}(x,x),\]
as $n\rightarrow\infty$, i.e. for every $\varepsilon>0$, there exists a set $F_{\varepsilon}\in\mathcal{F}$ such that $\|T^{t_{n}}_{U,W}(y^{(n)})-x\|<\varepsilon$ and $\|y^{(n)}-x\|<\varepsilon$ for any $n\in F_{\varepsilon}$.

Since $\mathcal{F}$ is finitely invariant, we may assume that for every $n\in F_{\varepsilon}$, $n>2J$.
Set
\[D^{(n)}_{j}:=y^{(n)}_{j}, G^{(n)}_{j}:=W_{j}W_{j-1}...W_{j-t_{n}+1}y^{(n)}_{j-t_{n}}U^{t_{n}}\]
for all $j\in[J]$.

Since $B_{0}(\mathcal{H})$ is an ideal of $B(\mathcal{H})$, we have  $\{D^{(n)}_{j}\}_{n\in\mathbb{N}}\subset B_{0}(\mathcal{H})$ and $\{G^{(n)}_{j}\}_{n\in\mathbb{N}}\subset B_{0}(\mathcal{H})$ for every $j\in[J]$.

And we get for any $n\in F_{\varepsilon}$,

\begin{eqnarray}\begin{split}\nonumber
\|D^{(n)}_{j}-P_{m}\|&=\|y^{(n)}_{j}-P_{m}\|\leq\|y^{(n)}-x\|<\varepsilon,\\
\|G^{(n)}_{j}-P_{m}\|&=\left\|\left(T^{t_{n}}_{U,W}(y^{(n)})\right)_{j}-P_{m}\right\|\\
&\leq\|T^{t_{n}}_{U,W}(y^{(n)})-x\|<\varepsilon,\\
\|W_{j+t_{n}}W_{j+t_{n}-1}...W_{j+1}D^{(n)}_{j}\|&=\left\|\left(T^{t_{n}}_{U,W}(y^{(n)})\right)_{j+t_{n}}U^{-t_{n}}\right\|\leq \|T^{t_{n}}_{U,W}(y^{(n)})-x\|<\varepsilon,\\
\|W^{-1}_{j-t_{n}+1}W^{-1}_{j-t_{n}+2}...W^{-1}_{j}G^{(n)}_{j}\|&=\|y^{(n)}_{j-t_{n}}U^{t_{n}}\|=\|y^{(n)}_{j-t_{n}}\|
\leq\|y^{(n)}-x\|<\varepsilon,
\end{split}\end{eqnarray}
which concludes the proof.

Now prove $(2)\Rightarrow(1)$. Let $V_{1}, V_{2}$ be non-empty open subsets of $\ell_{2}(\mathcal{A})$.
Assume that $F$ denotes the set of all elements $x=(x_{j})_{j}\in \ell_{2}(\mathcal{A})$ such that for some $J,m\in \mathbb{N}$, $x_{j}=P_{m}x_{j}$ for all $j\in [J]$ and $x_{j}=0$ for all $j\in \mathbb{Z}\setminus [J]$.
Since $F$ is dense in $\ell_{2}(\mathcal{A})$ by Proposition $2.2.1$ in $\cite{MT05}$, we can find $x=(x_{j})_{j}\in V_{1}\cap F$ and $y=(y_{j})_{j}\in V_{2}\cap F$. Then there exist $J,m\in\mathbb{N}$ such that $x_{j}=y_{j}=0$ for all $j\in\mathbb{Z}\setminus[J]$ and $x_{j}=P_{m}x_{j}$, $y_{j}=P_{m}y_{j}$ for all $j\in[J]$. Let $\|x\|=a$, $\|x\|=b$. Choose the sequences $\{D^{(n)}_{j}\}_{n\in\mathbb{N}}$ and $\{G^{(n)}_{j}\}_{n\in\mathbb{N}}$ for $j\in[J]$ satisfying $(2)$. Let us define the sequences $\{u_{n}\}_{n},\{v_{n}\}_{n}\subset\ell_{2}{(\mathcal{A})}$ by $(u_{n})_{j}:=D^{(n)}_{j}x_{j}, (v_{n})_{j}:=G^{(n)}_{j}y_{j}$ for $j\in [J]$ and $(u_{n})_{j}:=0, (v_{n})_{j}:=0$ for $j\in\mathbb{Z}\setminus[J]$. Set
\[\eta_{n}:=u_{n}+S^{t_{n}}_{U,W}v_{n}.\]

It is sufficient to prove that
\[\left(\eta_{n}, T^{t_{n}}_{U,W}(\eta_{n})\right)\stackrel{\mathcal F}{\longrightarrow}(x,y) \textrm{ as } n\rightarrow\infty.\]

From the condition $(2)$, for $\varepsilon>0$ there exists a set $F_{\varepsilon}\in\mathcal{F}$ such that for $n\in F_{\varepsilon}$ and $j\in[J]$, $\|D^{(n)}_{j}-P_{m}\|<\varepsilon/2(2J+1)a$, $\|G^{(n)}_{j}-P_{m}\|<\varepsilon/2(2J+1)b$, $\|W_{j+t_{n}}W_{j+t_{n}-1}...W_{j+1}D^{(n)}_{j}\|<\varepsilon/2(2J+1)a$ and $\|W^{-1}_{j-t_{n}+1}W^{-1}_{j-t_{n}+2}...W^{-1}_{j}G^{(n)}_{j}\|<\varepsilon/2(2J+1)b$.

If $n\in F_{\varepsilon}$ and $j\in [J]$, then we have that
\begin{eqnarray}\begin{split}\nonumber
\|(\eta_{n})_{j}-x_{j}\|&=\|(u_{n})_{j}+(S^{t_{n}}_{U,W}v_{n})_{j}-x_{j}\|\\
&=\|(u_{n})_{j}-x_{j}\|=\|D^{(n)}_{j}x_{j}-P_{m}x_{j}\|\\
&\leq\|D^{(n)}_{j}-P_{m}\|\|x\|<\varepsilon/2(2J+1).
\end{split}\end{eqnarray}
If $n\in F_{\varepsilon}$ and $j\in [J]-t_{n}$, then we have that
\begin{eqnarray}\begin{split}\nonumber
\|(\eta_{n})_{j}-x_{j}\|&=\|(u_{n})_{j}+(S^{t_{n}}_{U,W}v_{n})_{j}-x_{j}\|\\
&=\|(S^{t_{n}}_{U,W}v_{n})_{j}\|=\|W^{-1}_{j+1}W^{-1}_{j+2}...W^{-1}_{j+t_{n}}G^{(n)}_{j+t_{n}}y_{j+t_{n}}U^{-t_{n}}\|\\
&\leq \|W^{-1}_{j+1}W^{-1}_{j+2}...W^{-1}_{j+t_{n}}G^{(n)}_{j+t_{n}}\|\|y\|<\varepsilon/2(2J+1).
\end{split}\end{eqnarray}
If $n\in F_{\varepsilon}$ and $j\notin [J], [J]-t_{n}$, then we have that
\begin{eqnarray}\begin{split}\nonumber
\|(\eta_{n})_{j}-x_{j}\|&=\|(u_{n})_{j}+(S^{t_{n}}_{U,W}v_{n})_{j}-x_{j}\|\\
&=0.
\end{split}\end{eqnarray}
Therefore it follows that
\[\|\eta_{n}-x\|\leq\sum_{j\in\mathbb{N}}\|(\eta_{n})_{j}-x_{j}\|<\varepsilon\]
for all $n\in F_{\varepsilon}$.

Next, if $n\in F_{\varepsilon}$ and $j\in [J]$,then  we have that
\begin{eqnarray}\begin{split}\nonumber
\|(T^{t_{n}}_{U,W}\eta_{n})_{j}-y_{j}\|&=\|W_{j}W_{j-1}...W_{j-t_{n}+1}(u_{n})_{j-t_{n}}U^{t_{n}}+(v_{n})_{j}-y_{j}\|\\
&=\|(v_{n})_{j}-y_{j}\|=\|G^{(n)}_{j}y_{j}-P_{m}y_{j}\|\\
&\leq\|G^{(n)}_{j}-P_{m}\|\|y\|<\varepsilon/2(2J+1).
\end{split}\end{eqnarray}
If $n\in F_{\varepsilon}$ and $j\in [J]+t_{n}$, then we have that
\begin{eqnarray}\begin{split}\nonumber
\|(T^{t_{n}}_{U,W}\eta_{n})_{j}-y_{j}\|&=\|W_{j}W_{j-1}...W_{j-t_{n}+1}(u_{n})_{j-t_{n}}U^{t_{n}}+(v_{n})_{j}-y_{j}\|\\
&=\|W_{j}W_{j-1}...W_{j-t_{n}+1}(u_{n})_{j-t_{n}}U^{t_{n}}\|=\|W_{j}W_{j-1}...W_{j-t_{n}+1}D^{(n)}_{j-t_{n}}x_{j-t_{n}} U^{t_{n}}\|\\
&\leq\|W_{j}W_{j-1}...W_{j-t_{n}+1}D^{(n)}_{j-t_{n}}\|\|x\|<\varepsilon/2(2J+1).
\end{split}\end{eqnarray}
If $n\in F_{\varepsilon}$ and $j\notin [J], [J]+t_{n}$, then we have that
\begin{eqnarray}\begin{split}\nonumber
\|(T^{t_{n}}_{U,W}\eta_{n})_{j}-y_{j}\|&=\|W_{j}W_{j-1}...W_{j-t_{n}+1}(u_{n})_{j-t_{n}}U^{t_{n}}+(v_{n})_{j}-y_{j}\|\\
&=0.
\end{split}\end{eqnarray}

Therefore we get that
\[\|T^{t_{n}}_{U,W}\eta_{n}-y\|\leq\sum_{j\in\mathbb{N}}\|(T^{t_{n}}_{U,W}\eta_{n})_{j}-y_{j}\|<\varepsilon\]
for all $n\in F_{\varepsilon}$.
We complete the proof.

\end{proof}

$\mathcal{F}$-transitivity for weighted shifts on $\ell^{p}(\mathbb{Z})$ is already investigated in \cite{BMPP19} and Proposition \ref{pro3.1} generalizes the result of $\mathcal{F}$-transitivity for weighted shifts on $\ell^{p}$.

Next, we give a sufficient condition for $\mathcal{F}$-transitivity of weighted shift operator $T_{U, W}$ on $\ell_{2}(\mathcal{A})$. This extends the result of Theorem 3.2 in \cite{I24} to $\mathcal{F}$-transitivity.

\begin{theorem}\label{the3.2}
Let $\mathcal{F}$ be a finitely invariant Furstenberg family. And let $B_{\mathcal{H}}$ be the unit ball of $\mathcal{H}$ and $(t_{n})_{n}$ be an unbounded sequence of nonnegative integers. We denote $T_{U,W,n}:=T^{t_{n}}_{U,W}$ for all $n\in \mathbb{N}$. Suppose that for every $J\in\mathbb{N}$, there exist dense subsets $H^{(1)}_{j}$ and $H^{(2)}_{j}$ of $B_{\mathcal{H}}$ for all $j\in[J]$ such that
\[(W_{j+t_{n}}W_{j+t_{n}-1}...W_{j+1}, W^{-1}_{j-t_{n}+1}W^{-1}_{j-t_{n}+2}...W^{-1}_{j})\stackrel{\mathcal F}{\rightrightarrows}(0,0), \textrm{ as } n\rightarrow \infty,\]
\[ \textrm{ pointwise on } H^{(1)}_{j}\times H^{(2)}_{j}, \textrm{ for } j\in[J],\]
 that is, for any $\varepsilon>0$, there exists a set $F_{\varepsilon}\in\mathcal{F}$ such that for any $n\in F_{\varepsilon}$, $j\in [J]$ and $(x,y)\in H^{(1)}_{j}\times H^{(2)}_{j}$,
$\|W_{j+t_{n}}W_{j+t_{n}-1}...W_{j+1}x\|<\varepsilon, \|W^{-1}_{j-t_{n}+1}W^{-1}_{j-t_{n}+2}...W^{-1}_{j}y\|<\varepsilon$.
Then $(T_{U,W,n})_{n}$ is $\mathcal{F}$-transitive on $\ell_{2}(\mathcal{A})$.
\end{theorem}

\begin{proof}
Assume that $J,m\in \mathbb{N}$. Since for each $j\in\mathbb [J]$, $H^{(1)}_{j}$ and $H^{(2)}_{j}$ are dense in $B_{\mathcal{H}}$, we can find sequences $(f^{j}_{i,l})_{i}\subset H^{(1)}_{j}$  and $(g^{j}_{i,l})_{i}\subset H^{(2)}_{j}$ such that $f^{j}_{i,l}\rightarrow e_{l}$ and $g^{j}_{i,l}\rightarrow e_{l}$ as $i\rightarrow \infty$ for all $j\in[J]$ and $l\in [m]$.
For each $j\in[J]$ define the operators $D^{(i)}_{j}$ and $G^{(i)}_{j}$ by
\begin{displaymath}
D^{(i)}_{j}(e_{l}):=\left\{ \begin{array}{ll}
 f^{(j)}_{i,l}, &  l\in[m]\\
 0, & l \notin [m]
\end{array} \right.
\textrm{ and     }
G^{(i)}_{j}(e_{l}):=\left\{ \begin{array}{ll}
 g^{(j)}_{i,l}, &  l\in[m]\\
 0, & l \notin [m].
\end{array} \right.
\end{displaymath}

Then trivially  $\|D^{(i)}_{j}-P_{m}\|\rightarrow 0$ and $\|G^{(i)}_{j}-P_{m}\|\rightarrow 0$, as $i\rightarrow \infty$ for any $j\in[m]$.
And for any $n\in F_{\varepsilon/(2m+1)}$, we have that

\begin{eqnarray}\begin{split}\nonumber
\|W_{j+t_{n}}W_{j+t_{n}-1}...W_{j+1}D^{(n)}_{j}\|&\leq \sum_{l\in[m]}\|W_{j+t_{n}}W_{j+t_{n}-1}...W_{j+1}D^{(n)}_{j}e_{l}\| \\ &=\sum_{l\in[m]}\|W_{j+t_{n}}W_{j+t_{n}-1}...W_{j+1}f^{(j)}_{n,l}\|<\varepsilon,\\
 \|W^{-1}_{j-t_{n}+1}W^{-1}_{j-t_{n}+2}...W^{-1}_{j}G^{(n)}_{j}\|&\leq \sum_{l\in[m]}\|W^{-1}_{j-t_{n}+1}W^{-1}_{j-t_{n}+2}...W^{-1}_{j}G^{(n)}_{j}e_{l}\|\\
&=\sum_{l\in[m]}\|W^{-1}_{j-t_{n}+1}W^{-1}_{j-t_{n}+2}...W^{-1}_{j}g^{(j)}_{n,l}\|<\varepsilon
\end{split}\end{eqnarray}
which concludes the proof.
\end{proof}

\begin{corollary}\label{co3.3}
Let $\mathcal{F}$ be a finitely invariant Furstenberg family. And let $(t_{n})_{n}$ be an unbounded sequence of nonnegative integers. We denote $T_{U,W,n}:=T^{t_{n}}_{U,W}$ for all $n\in \mathbb{N}$.

Suppose that for every $J, m\in\mathbb{N}$,
\[\left(\|W_{j+t_{n}}W_{j+t_{n}-1}...W_{j+1}P_{m}\|, \|W^{-1}_{j-t_{n}+1}W^{-1}_{j-t_{n}+2}...W^{-1}_{j}P_{m}\|\right)\stackrel{\mathcal F}{\rightrightarrows}(0, 0), \textrm{ as } n\rightarrow \infty, \textrm{ for } j\in[J].\]
Then $T_{U,W}$ is $\mathcal{F}$-transitive.
\end{corollary}

Using the above corollary, we can find an example of a topologically mixing operator $T_{U,W}$ as follows;

\begin{example}\label{ex3.4}
Let $\mathcal{H}$ be a separable Hilbert space and $(e_{j})_{j\in\mathbb{Z}}\subset\mathcal{H}$ be an orthonormal basis for $\mathcal{H}$. Let $\alpha>1$ be a positive real number. And $V$ is a bounded linear operator on $\mathcal{H}$ defined by
\begin{displaymath}
V(e_{n})=\left\{ \begin{array}{ll}
 \alpha e_{n+1}, & \textrm{for } n<0\\
 \frac{1}{\alpha}e_{n+1}, & \textrm{for } n\geq 0.
\end{array} \right.
\end{displaymath}

We define $W_{j}=V$ for every $j\in \mathbb{Z}$. Then since
\[\lim_{k\rightarrow\infty}\|W_{j+k}W_{j+k-1},...,W_{j+1}P_{m}\|= \lim_{k\rightarrow\infty}\|V^{k}P_{m}\|=0\]
and
\[\lim_{k\rightarrow\infty}\|W^{-1}_{j-k+1}W^{-1}_{j-k+2}...W^{-1}_{j}P_{m}\|= \lim_{k\rightarrow\infty}\|V^{-k}P_{m}\|=0,\]
$T_{U,W}$ is topologically mixing from Corollary \ref{co3.3}
\end{example}

\end{document}